\documentclass[a4paper,11pt,reqno]{amsart}
\usepackage[T1]{fontenc}
\usepackage[utf8]{inputenc}
\usepackage{lmodern}
\usepackage{amsmath, amsthm, amssymb,amscd, mathrsfs, amsfonts, mathtools}
\usepackage{hyperref}
\usepackage{euler}
\usepackage{times}
\usepackage[all]{xy}
\usepackage{todonotes}
\usepackage{xcolor, color}
\usepackage{enumerate}
\usepackage{tabularx}

\usepackage[lite]{amsrefs}

\renewcommand{\PrintDOI}[1]{\href{http://dx.doi.org/\detokenize{#1}}{doi: \detokenize{#1}}%
	\IfEmptyBibField{pages}{, (to appear in print)}{}}

\newcommand{\Z}{\mathbb{Z}}

\newcommand{\rt}{\vartriangleright}

\newtheorem{theorem}{Theorem}[section]

\newtheorem{proposition}[theorem]{Proposition}

\newtheorem{definition}[theorem]{Definition}

\newtheorem{example}[theorem]{Example}

\newtheorem{remark}[theorem]{Remark}

\title{ Ternary and $n$-ary $f$-Distributive Structures   }

\author{Indu R. U. Churchill}
\address{Department of Mathematics,
	University of South Florida, Tampa, FL 33620, U.S.A.}
 \email{udyanganiesi@mail.usf.edu}

\author{M. Elhamdadi}
\address{Department of Mathematics,
	University of South Florida, Tampa, FL 33620, U.S.A.}
\email{emohamed@math.usf.edu}

 \author{M. Green}
 \address{Department of Mathematics,
 	University of South Florida, Tampa, FL 33620, U.S.A.}
 \email{mjgreen@mail.usf.edu}

 \author{A. Makhlouf}
 \address{Universit\'e de Haute Alsace, Laboratoire de Math\'ematiques, Informatique et Applications,  France}
 \email{Abdenacer.Makhlouf@uha.fr}

\begin{document}

\begin{abstract}

\noindent We introduce and study ternary $f$-distributive structures, Ternary $f$-quandles and more generally their higher $n$-ary analogues.  A classification of ternary $f$-quandles is provided in low dimensions.  
Moreover, we study extension theory and  introduce a cohomology theory for ternary, and more generally $n$-ary, $f$-quandles. Furthermore, we give some computational examples.
\end{abstract}

\maketitle


\tableofcontents
\section{Introduction}
The first instances of ternary operations appeared in the nineteenth century when Cayley considered cubic matrices. 
Ternary operations or more generally  $n$-ary operations appeared  naturally  in  various  domains  of  theoretical  and  mathematical
physics.  The first instances of ternary Lie algebras appeared  in the Nambu's Mechanics when generalizing hamiltonian mechanics by considering more than one hamiltonian \cite{Nambu}. The algebraic formulation of Nambu's Mechanics was achieved  by Takhtajan in \cite{MR1290830}. Moreover, ternary algebraic structures appeared in String and Superstring theories when Basu and Harvey suggested to replace Lie algebra in the context of Nahm equations by a 3-Lie algebra. Furthermore, a ternary operation was used by Bagger-Lambert in the context of Bagger-Lambert-Gustavsson model of M2-branes and in the
 work of Okubo \cite{Okubo} on Yang-Baxter equation which gave impulse to significant development on $n$-ary algebras.
In recent years, there has been a growth of interests in many generalizations of binary structures to higher $n$-ary contexts. 
In Lie algebra theory, for example, the bracket is replaced by a $n$-ary bracket and the Jacobi identity is replaced by its higher analogue, see \cite{Filippov}. Generalizations of quandles to the ternary case were done recently in \cite{EGM}.

In this paper we introduce and study a twisted version of ternary, respectively  $n$-ary, generalizations of racks and quandles, where the structure is defined by a ternary operation and a linear map twisting  the distributive property. These type of algebraic structures, called sometimes Hom-algebras,  appeared first in quantum deformations of algebras of vector fields, motivated by physical aspects. A systematic study and mathematical aspects were provided for Lie type algebras by Hartwig-Larsson and Silvestrov in \cite{HLS} whereas associative and other nonassociative algebras were discussed by the fourth author  and Silvestrov in \cite{MS} and $n$-ary Hom-type algebras in \cite{AMS}. The main feature of all these generalization is that the usual identities are twisted by a homomorphism. 
We introduce in this article  the notions of ternary, respectively  $n$-ary, $f$-shelf (resp. $f$-rack, $f$-quandle), give some key constructions and properties. Moreover, we provide  a classification in low dimensions of $f$-quandles.  We also study extensions  and modules, as well as cohomology  theory of these structures.  For classical quandles theory, we refer to 	\cite{EN}, see also \cite{And-Grana,Carter-Crans-Elhamdadi-Saito,Carter-Elhamdadi-Grana-Saito,CES,CENS,EMR,FR,Joyce,Matveev}. 
	For basics and some development of  Hom-type algebras, we refer to \cite{FG,LS,MS,MS4,MakhloufSilv,Yau1,Yau2,Yau3}

 This paper is organized as follows. In Section 2, we review the basics of $f$-quandles and ternary distributive structures and give the general $n$-ary setting.  In Section 3, we discuss a key construction introduced by Yau, we show that given a $n$-ary $f$-shelf (resp. $f$-rack, $f$-quandle) and a shelf morphism then one constructs a new  $n$-ary $f$-shelf (resp. $f$-rack, $f$-quandle) and we provide examples.  In Section 4, we provide a   classification of ternary $f$-quandles in low dimensions.  Section 5 gives the extension theory of $f$-quandles and modules.  Finally, in Section 6 we introduce the cohomology of $n$-ary $f$-distributive structures and give examples.  

\section{$f$-Quandles and Ternary (resp. $n$-ary)  Distributive Structures}
In this section we aim to introduce the notion of ternary and more generally $n$-ary $f$-quandles, generalizing the notion of $f$-quandle given in  \cite{CEGM}.
\subsection{A Review of $f$-Quandles and related structures}
First, we review the basics of the binary $f$-quandles. We refer to  \cite{CEGM} for the complete study. Classical theory of quandle could be found in 
\begin{definition}\label{twistdef}
An {\it $f$-shelf} is a triple $(X, *, f)$ in which $X$ is a set, $*$ is a binary operation on $X$, and $f \colon X \to X$ is a map such that, for any $x,y,z \in X$, the identity\\
\begin{center} \label{TtwistedCon-I}
$(x *y) * f(z) =  (x * z) * (y * z)$\\
 \end{center}
holds.  An {\it $f$-rack} is an $f$-shelf such that, for any $x,y \in X$, there exists a unique $z \in X$ such that\\
\begin{center} \label{TtwistedCon-II}
$z* y = f(x).$
\end{center}

An {\it $f$-quandle} is an $f$-rack such that, for each $x \in X$, the identity \\
\begin{center} \label{TtwistedCon-III}
$x* x =f(x)$
\end{center}
holds.\\

An \textit{ $f$-crossed set} is an $f$-quandle $(X, *, f)$ such that  $f: X \to X$ satisfies $x * y =f(x) $ whenever   $ y* x = f(y)$ for any $x, y \in X$.
\end{definition}



\begin{definition}
Let $(X_1, *_1, f_1)$ and $(X_2, *_2, f_2)$ be two $f$-racks (resp. $f$-quandles). A map $\phi : X_1 \to X_2$ is an $f$-rack (resp. $f$-quandle) morphism if it satisfies $\phi (a *_1 b) = \phi(a) *_2 \phi(b)$ and $\phi \circ f_1 = f_2 \circ \phi$.
\end{definition}

\begin{remark}
A category of $f$-quandles is a category whose objects are tuples $(A, *, f)$ which are $f$-quandles and morphism are $f$-quandle morphisms.
\end{remark}

\label{Examples}
Examples of $f$-quandles include the following:

  \begin{itemize}
  \setlength{\itemsep}{-3pt}

\item
 Given any set  $X$ and map $f: X \to X$, then the operation  $ x * y = f(x)$ for any $ x, y \in X$ gives a $f$-quandle. We call this a \textit{trivial} $f$-quandle structure on $X$.\\

\item
For any group $G$ and any group endomorphism $f$ of $G$, the operation $x *  y=y^{-1} x f(y) $ defines a $f$-quandle structure on $G$.\\

\item

Consider the Dihedral quandle $R_n$, where $ n \geq 2$, and let $f$ be an automorphism of $R_n$.  Then $f$ is given by $f(x)=ax+b$, for some invertible element $a \in \mathbb{Z}_n $ and some $ b \in \mathbb{Z}_n$ \cite{EMR}.  The binary operation
$x  *  y =f(2y-x)= 2ay- ax +b  \pmod{n}$ gives a $f$-quandle structure called the  {\it Dihedral  $f$-quandle}.\\
\item
Any ${\Z }[T^{\pm 1}, S]$-module $M$
is a $f$-quandle with
$x * y=Tx+Sy$, $x,y \in M$ with $ TS = ST$ and $f(x)=(S+T)x$, called an {\it  Alexander  $f$-quandle}.
  \end{itemize}

\begin{remark}
Axioms 
of Definition \ref{twistdef} give the following identity,  $$(x * y) * ( z * z)  =  (x*z) * (y * z) .$$  We note that the two medial terms in this equation are swapped (resembling the mediality condition of a quandle).  Note also that the mediality in the general context may not be satisfied for $f$-quandles. For example one can check that the $f$-quandle given in item (2) of Examples 
is not medial.
\end{remark}


\subsection{ Ternary and $n$-ary  $f$-Quandles }

\noindent Now we introduce and discuss the analogous notion of a $f$-quandle in the ternary setting and more generally in the $n$-ary setting.

\begin{definition}\label{twistedternarydef}  {\rm 
Let $Q$ be a set, $\alpha$ be a morphism and $T: Q\times Q \times Q \rightarrow Q$ be a ternary operation on $Q$.  The operation  $T$ is said to be {\it right $f$-distributive} with respect to $\alpha$ if it satisfies the following condition
for all  $ x,y,z,u,v \in Q $
\begin{equation}\label{TC1-ternary}T(T(x,y,z),\alpha(u),\alpha(v))=T(T(x,u,v),T(y,u,v),T(z,u,v)).\end{equation}
}
The previous condition is called right $f$-distributivity.
 \end{definition}

\begin{remark} {\rm
Using the diagonal map $D: Q \rightarrow Q \times Q \times Q=Q^{\times 3} $ such that $D(x)=(x,x,x)$, equation  (\ref{TC1-ternary}) can be written, as a map from $Q^{\times 5} $ to $Q$,  in the following form
\begin{equation}\label{Teq2}
T \circ (T \times \alpha \times \alpha)=T \circ ( T \times  T \times T )\circ\rho\circ( id \times id \times id \times D \times D),
\end{equation}
where $id$ stands for the identity map. In the whole paper we denote by
$\rho:Q^{\times 9}\rightarrow Q^{\times 9}$ the map defined as $\rho=p_{6,8}\circ p_{3,7}\circ p_{2,4}$  where $p_{i,j}$ is the transposition $i^{th}$ and $j^{th}$ elements, i.e.
\begin{equation}\label{RHO}
\rho(x_1,\cdots,x_9)=(x_1,x_4,x_7,x_2,x_5,x_8,x_3,x_6,x_9).
\end{equation}
}
 \end{remark}
 \begin{definition}\label{twistedternaryquandledef}  {\rm 
Let $T: Q\times Q \times Q \rightarrow Q$ be a  ternary operation on a set $Q$.
 The triple $(Q,T,\alpha)$ is said to be a \emph{ternary $f$-shelf} if the  identity (\ref{TC1-ternary}) holds.  If, in addition, for all $a,b \in Q$, the map $R_{a,b}:Q \rightarrow Q$ given by $R_{a,b}(x)=T(x,a,b)$ is invertible,  
  then $(Q,T,\alpha)$ is said to be a \emph{ternary $f$-rack}.  If further $T$ satisfies
$  T(x,x,x)=\alpha(x),$ for all $ x \in Q,$
  then $(Q,T, \alpha)$ is called a \emph{ternary $f$-quandle}.
}
\end{definition}
\begin{remark}
Using the right translation $R_{a,b} : Q \to Q$ defined as $R_{a,b}(x) = T(x, a, b)$, the identity \eqref{TC1-ternary} can be written as  $ R_{\alpha(u), \alpha(v)} \circ R_{y, z} = R_{R_{u,v}(y), R_{u,v}(z)} \circ R_{u,v}$.
\end{remark}
 \begin{example} \label{T-stepITheorem}
 Let $(Q,*,f)$ be a $f$-quandle and define a ternary operation on $Q$ by $$T(x,y,z)=(x * y)* f(z), \forall x,y,z \in Q.$$  It is straightforward to see that $(Q, T, \alpha)$ is a ternary $f$-quandle where $\alpha = f\circ f$.  Note that in this case $R_{a,b}=R_b \circ R_a$. We will say that this ternary $f$-quandle is induced by a (binary) quandle.
 \end{example}

\begin{remark} Binary $f$-quandles relates to ternary $f$-quandle when $\alpha=f\circ f$. Also, ternary operations $T$  lead to binary operations by setting for example $ x*y = T (x,y,y)$.
\end{remark}

   \begin{example}\label{TAlex}
   Let $(M,*,f)$ be an Alexander $f$-quandle, then the ternary $f$-quandle $(M, T ,\alpha )$ coming from $M$,  has the operation $T(x,y,z)=Px+Qy+Rz$ where $P, Q, R$ commutes with each other and $\alpha (x) = (P + Q + R)x$.
    \end{example}

 \begin{example}
 Any group $G$ with the ternary operation $T(x,y,z)=f(xy^{-1}z)$  gives an example of ternary $f$-quandle.  This is called $f$ $heap$ (sometimes also called a groud) of the group $G$.
\end{example}

\noindent
A morphism of ternary quandles is a map $\phi: ( Q, T) \rightarrow (Q',T')$ such that $$\phi(T(x,y,z))=T'(\phi(x),\phi(y),\phi(z)).$$  A bijective ternary quandle endomorphism is called ternary $f$-quandle automorphism.\\
Therefore, we have a category 
whose objects are ternary $f$-quandles and morphisms as defined above.


 As in the case of the binary quandle there is a notion of {\it medial} ternary quandle

 \begin{definition}  {\rm \cite{Borowic}
A ternary quandle $(Q,T, \alpha)$  is said to be {\it medial} if for all $a, b, c, d, e, f, g, h, k  \in Q,$ the following identity is satisfied
$$T(T(a,b,c),T(d,e,f),T(g,h,k))= T(T(a,d,g),T(b,e,h),T(c,f,k)).$$
}
 \end{definition}

  This definition of {\it mediality} can be written in term of the following commutative diagram

\[\begin{xy}
    (-2, 20)*+{\underbrace{Q \times \cdot \cdot \cdot \times Q}_{\text{ 9 times}}}="1";
    (-30,10)*+{Q \times Q \times Q}="2";
    (-30,-10)*+{Q}="3";
    (30,-10)*+{Q \times Q \times Q}="5";
    (30, 10)*+{ \underbrace{Q \times \cdot \cdot \cdot \times Q}_{\text{ 9 times}} }="6";
        {\ar^{\rho} "1";"6"};
        {\ar_{T \times T \times T} "1";"2"};
        {\ar_{T} "2";"3"};
        {\ar^{T \times T \times T} "6";"5"};
         {\ar_{T} "5";"3"};

\end{xy}
\]


%
%
%


  \begin{example}
  Every affine ternary $f$-quandle is medial.       \end{example}


We generalize the notion of ternary $f$-quandle to $n$-ary setting.

 \begin{definition}\label{n-Tternarydef}  {\rm 
An $n$-ary distributive set is a triple $(Q,T, \alpha)$ where $Q$ is a set, $\alpha : Q \to Q$ is a morphism and $T: Q^{ \times n} \rightarrow Q$ is an $n$-ary operation satisfying the following conditions:
\begin{enumerate}[(I)]
\item \begin{eqnarray*}\label{TC1-n-ary}& T(T(x_1,\cdots,x_n),\alpha(u_1),\cdots, \alpha(u_{n-1}))=\nonumber \\ & T(T(x_1,u_1,\cdots, u_{n-1}),T(x_2,u_1,\cdots, u_{n-1}),\cdots, T(x_n,u_1,\cdots, u_{n-1})),\end{eqnarray*} $ \forall x_i,u_i \in Q$ ($f$-distributivity).\\

\item For all $a_1,\cdots,a_{n-1} \in Q$, the map $R_{a_1,\cdots,a_{n-1}}:Q \rightarrow Q$ given by $$R_{a_1,\cdots,a_{n-1}}(x)=T(x,a_1,\cdots,a_{n-1})$$ is invertible. \\ 

\item For all $x \in Q$,\begin{equation*}\label{TC3-n-ary} T(x,\cdots,x)=\alpha(x).\end{equation*}

\end{enumerate}

If $T$ satisfies only condition (1), then $(Q,T, \alpha)$ is said to be an $n$-ary $f$-shelf.  If both conditions (1) and  (2) are satisfied then $(Q,T, \alpha)$ is said to be an  $n$-ary $f$-rack. If all three conditions (1), (2) and (3) are satisfied then $(Q,T, \alpha)$ is said to be an $n$-ary $f$-quandle.

}
 \end{definition}
 \begin{example} \label{Tn-ary}
 Let $(Q,*,f)$ be an $f$-quandle and define an $n$-ary twisted operation on $Q$ by $$T(x_1, x_2, \cdots, x_n)=(( \dots (x_1 * x_2) * f(x_3)) * f^2(x_4))* \dots ) *  f^{n-2}(x_n) \forall x_i \in Q,$$ where $ i = 1, \cdots, n$. \\ It is straightforward to see that $(Q, T, \alpha)$ is an $n$-ary $f$-quandle where $\alpha = f^{n-1}$.

 \end{example}

 \begin{example}
 Let $(M,*,f)$ be an Alexander $f$-quandle, then the $n$-ary $f$-quandle $(M, T ,\alpha )$ coming from $M$ has the operation $T(x_1, \cdots, x_n)=S_1 x_1+S_2 x_2+ \cdots +S_n x_n,$ where $S_i, i= 1, \cdots , n$ commutes with each other and $\alpha (x) = (S_1 + S_2 + \cdots + S_n)x$.

 \end{example}

 \begin{definition}  {\rm
An $n$-ary quandle $(Q,T, \alpha)$  is said to be {\it medial} if for all $x_{ij}  \in Q, 1 \leq i,j \leq n,$ the following identity is satisfied
\begin{eqnarray*}
T(T(x_{11},x_{12},\cdots, x_{1n}),T(x_{21},x_{22},\cdots, x_{2n}),\cdots, T(x_{n1},x_{n2},\cdots, x_{nn}))=\nonumber \\
T(T(x_{11},x_{21},\cdots, x_{n1}),T(x_{12},x_{22},\cdots, x_{n2}),\cdots, T(x_{1n},x_{2n},\cdots, x_{nn})).
\end{eqnarray*}
}
 \end{definition}

\section{Yau Twist}

 The following  proposition provide a way of constructing  new  $n$-ary $f$-shelf (resp. $f$-rack, $f$-quandle) along a shelf morphism. In particular, given an $n$-ary  shelf (resp.  rack, quandle) and a  shelf morphism, one may obtain  an $n$-ary $f$-shelf (resp. $f$-rack, $f$-quandle). Recall that this construction was introduced first by Yau to deform a Lie algebra to a Hom-Lie algebra along a Lie algebra morphism. It was generalized to different situation, in particular to $n$-ary algebras in \cite{AMS}.

\begin{proposition}
Let $(Q,T, \alpha)$ be an  $n$-ary $f$-shelf (resp. $f$-rack, $f$-quandle) and $\beta : Q\rightarrow Q$ a shelf morphism then  $(Q,\beta\circ T, \beta\circ \alpha)$  is a new   $n$-ary $f$-shelf (resp. $f$-rack, $f$-quandle).
\end{proposition}
\begin{proof}
\begin{enumerate} [(A)]
  \item \begin{eqnarray*}
 \lefteqn{
 LHS (I) = \beta \circ T(\beta \circ T(x_1,\cdots,x_n),\beta \circ \alpha(u_1),\cdots, \beta \circ \alpha(u_{n-1}))} \nonumber \\ && =
\beta \circ T(T(\beta(x_1), \cdots, \beta(x_n)), \alpha(\beta(u_1)), \cdots, \alpha(\beta(u_{n-1}))) \nonumber \\ && =
T(\beta (T(\beta(x_1), \cdots, \beta(x_n))), \beta (\alpha(\beta(u_1))), \cdots, \beta ( \alpha(\beta(u_{n-1})))) \nonumber \\ && =
T(T(\beta^2(x_1), \cdots, \beta^2(x_n)), \alpha(\beta^2(u_1)), \cdots, \alpha(\beta^2(u_{n-1}))) \nonumber \\ && =
T[T\{\beta^2(x_1), \beta^2(u_1), \cdots, \beta^2(u_{n-1})\}, T\{\beta^2(x_2), \beta^2(u_1), \cdots, \beta^2(u_{n-1})\},\nonumber \\ &&  \quad \cdots,  T\{\beta^2(x_n), \beta^2(u_1), \cdots, \beta^2(u_{n-1})\}] \nonumber \\ && =
T[T\{\beta^2(x_1, u_1, \cdots, u_{n-1})\}, T\{\beta^2(x_2, u_1, \cdots, u_{n-1})\},  \\ && \quad  \cdots,  T\{\beta^2(x_n, u_1, \cdots, u_{n-1})\}]  \quad \because [I] \nonumber \\ && =
\beta \circ T[ T\{\beta(x_1, u_1, \cdots, u_{n-1})\},T\{\beta(x_2, u_1, \cdots, u_{n-1})\} , \\ && \quad  \cdots,  T\{\beta(x_n, u_1, \cdots, u_{n-1})\}] \nonumber \\&& =
 \beta \circ T(\beta \circ T(x_1,u_1,\cdots, u_{n-1}),\beta \circ T(x_2,u_1,\cdots, u_{n-1}),\\ && \quad \cdots, \beta \circ T(x_n,u_1,\cdots, u_{n-1})) \nonumber \\&& =
 RHS (I) , \ \ \  \forall x_i,u_i \in Q.
 \end{eqnarray*}
  \noindent Therefore, $(Q, \beta \circ T, \beta \circ \alpha)$ is an $n$-ary $f$-shelf.\\

 \item  For all $a_1,\cdots,a_{n-1} \in Q$, the map $R_{\beta(a_1),\cdots,\beta(a_{n-1})}:Q \rightarrow Q$ given by \\
 \begin{eqnarray*}
 \lefteqn{
 R_{\beta(a_1),\cdots,\beta(a_{n-1})}(x)} \nonumber \\ && =
 T(\beta(x), \beta(a_1), \cdots, \beta(a_{n-1})) \nonumber \\ && =
 \beta \circ T(x,a_1,\cdots,a_{n-1})\end{eqnarray*} is invertible.  \\
 \noindent Therefore, $(Q, \beta \circ T, \beta \circ \alpha)$ is an $n$-ary $f$-rack.\\

 \item  For all $x \in Q$,\\
 \begin{eqnarray*}
 \lefteqn{
 \beta \circ  T(x,\cdots,x)} \nonumber \\ &&=
 T(\beta(x), \cdots, \beta(x)) \nonumber \\ &&=
\beta \circ \alpha(x).\end{eqnarray*}\\
 \noindent Therefore, $(Q, \beta \circ T, \beta \circ \alpha)$ is an $n$-ary $f$-quandle.
 \end{enumerate}
\end{proof}
\begin{example}
Let $(Q,*)$ be a quandle and define a ternary operation on $Q$ by $T(x,y,z) = (x * y) * z,\forall x,y,z \in Q$. It is straightforward to see that (Q,T) is a ternary quandle. Let $\alpha : Q \to Q$ be a morphism. Then $( Q, \alpha \circ T, \alpha)$ is a ternary $f$-quandle.
\end{example}

 \begin{example}
 Let $(M,T)$ be an Alexander ternary quandle, where $T(x,y,z)=t^2x+t(1-t)y+(1-t)z$. Let $f: M \to M$ be a morphism. Then $(M, f \circ T, f)$ is a ternary $f$-quandle.
 \end{example}
 \begin{example}
   Let $M$ be a $\mathbb{Z}[t,t^{-1}]$-module. Consider the   ternary quandle $(M,T)$ defined by  the operation $T(x,y,z)=t^2x+t(1-t)y+(1-t)z$, where $x,y,z\in M$. For any map $\phi :M\rightarrow M$ satisfying $\phi \circ T= T\circ ( \phi \otimes\phi \otimes\phi )$, we associate a new ternary $f$-quandle $(M,\phi \circ T, \phi) $.     \end{example}

 \begin{example}
  Let $M$ be any ${\Lambda}$-module where
${\Lambda}=\mathbb{Z}[t^{\pm 1},s]$. The operation $T(x,y,z)=tx+sy+kz$ where $\alpha(x) = (t+s+k)x$ defines a ternary $f$-quandle structure on $M$.  (We call this an {\it affine} ternary $f$-quandle). Let $f: M \to M$ be a morphism. Then $(M, f \circ T, f \circ \alpha)$ is a ternary $f$-quandle.
    \end{example}
\begin{example}
 Let $(Q,*)$ be a quandle and define a $n$-ary  operation on $Q$ by $T(x_1, x_2, \cdots, x_n)=(( \dots (x_1 * x_2) * x_3 )* x_4)* \dots ) * x_n), \forall x_i \in Q$ where $ i = 1, \cdots, n$. Let $f: Q \to Q$ be a morphism. It is straightforward to see that $(Q,  f \circ T, f )$ is an $n$-ary $f$-quandle.
 \end{example}

\section{Classification of Ternary $f$-Quandles of low orders}

\noindent We developed a simple program to compute all ternary $f$-quandles of orders 2 and 3.  The results of which we used to obtain the complete list of isomorphism classes.

\noindent For order 2, we found 6 distinct isomorphism classes, each of which can be defined over $\mathbb{Z}_2$ by one of the following maps: $\tau(x,y,z)=x$, $\tau(x,y,z)=x+1$, $\tau(x,y,z)=x+y$, $\tau(x,y,z)=x+z$, $\tau(x,y,z)=x+y+z$, or $\tau(x,y,z)=x+y+z+1$.

In the case of order 3, we found a total of 84 distinct isomorphism classes, including 30 ternary quandles (those such that $f=id_Q$).

Since for each fixed $a,b$, the map $x \mapsto \tau(x,a,b)$ is a permutation, in the following table we describe all ternary $f$-quandles of order three in terms of the columns of the Cayley table. Each column is a permutation of the elements and is described in standard notation, that is by explicitly writing it in terms of products of disjoint cycles.  Thus for a given $z$ we give the permutations resulting from fixing $y=1,2,3$. For example, the ternary set $\tau_{12}(x,y,z)$ with the Cayley Table \ref{terntable1} will be represented with the permutations $(1),(12),(13); (12),(1),(23); (13),(23),(1)$. This will appear on Table \ref{terntable3} as shown in Table \ref{terntable2}.

\begin{table}
\begin{center}
\begin{tabular}{|c|c|c||c|c|c||c|c|c|}
\hline
\multicolumn{3}{|c||}{z=1} &
\multicolumn{3}{|c||}{z=2} &
\multicolumn{3}{|c|}{z=3} \\ \hline
1 & 2 & 3 & 2 & 1 & 1 & 3 & 1 & 1 \\ \hline
2 & 1 & 2 & 1 & 2 & 3 & 2 & 3 & 2 \\ \hline
3 & 3 & 1 & 3 & 3 & 2 & 1 & 2 & 3 \\ \hline
\end{tabular}
\medskip
\caption{Cayley representation of ternary quandle $\tau_{12}$ }\label{terntable1}
\end{center}
\end{table}

\begin{table}
 \begin{center}
\begin{tabular}{|c|c|c|c|}
\hline
$\tau$ & z=1 & z=2 & z=3 \\ \hline
$\tau_{12}$ & (1),(12),(13)& (12),(1),(23)& (13),(23),(1)\\ \hline
\end{tabular}
\medskip
\caption{Permutation representation of ternary quandle $\tau_{12}$ }\label{terntable2}
\end{center}
\end{table} 

 Table \ref{terntable4} lists the isomorphism classes, the first table lists those that such that $f=id_Q$, and the second table lists classes with members with a non-trivial twisting.

\begin{center}
\begin{table}
\tiny
\begin{tabular}{|c|c|c||c|c|c|}
\hline
 z=1 & z=2 & z=3 &  z=1 & z=2 & z=3\\ \hline
(1),(1),(1) & (1),(1),(1) & (1),(1),(1) &
(1),(1),(1) & (1),(1),(1) & (1),(1),(1 2) \\ \hline
(1),(1),(1) & (1),(1),(1) & (1 2),(1 2),(1) &
(1),(1),(1) & (1),(1),(1) & (1 2),(1 2),(1 2) \\ \hline
(1),(1),(1) & (1),(1),(2 3) & (1),(2 3),(1) &
(1),(1),(1) & (2 3),(1),(1) & (2 3),(1),(1) \\ \hline
(1),(1),(1) & (2 3),(1),(2 3) & (2 3),(2 3),(1) &
(1),(1),(1 2) & (1),(1),(1 2) & (1),(1),(1 2) \\ \hline
(1),(1),(1 2) & (1),(1),(1 2) & (1 2),(1 2),(1) &
(1),(1),(1 2) & (1),(1),(1 2) & (1 2),(1 2),(1 2) \\ \hline
(1),(1),(1 2 3) & (1 2 3),(1),(1) & (1),(1 2 3),(1) &
(1),(1),(1 3 2) & (1 3 2),(1),(1) & (1),(1 3 2),(1) \\ \hline
(1),(1),(1 3) & (1),(1 3),(1) & (1 3),(1),(1) &
(1),(1),(1 3) & (1 3),(1),(1 3) & (1 3),(1),(1) \\ \hline
(1),(1),(1 3) & (1 3),(1 3),(1 3) & (1 3),(1),(1) &
(1),(2 3),(2 3) & (2 3),(1),(2 3) & (2 3),(2 3),(1) \\ \hline
(1),(2 3),(2 3) & (1 3),(1),(1 3) & (1 2),(1 2),(1) &
(1),(1 2),(1 2) & (1 2),(1),(1 2) & (1),(1),(1 2) \\ \hline
(1),(1 2),(1 2) & (1 2),(1),(1 2) & (1 2),(1 2),(1 2) &
(1),(1 2),(1 3) & (1 2),(1),(2 3) & (1 3),(2 3),(1) \\ \hline
(1),(1 2 3),(1 2 3) & (1 2 3),(1),(1 2 3) & (1 2 3),(1 2 3),(1) &
(1),(1 2 3),(1 3 2) & (1 3 2),(1),(1 2 3) & (1 2 3),(1 3 2),(1) \\ \hline
(1),(1 3 2),(1 2 3) & (1 2 3),(1),(1 3 2) & (1 3 2),(1 2 3),(1) &
(1),(1 3),(1 2) & (2 3),(1),(1 2) & (2 3),(1 3),(1) \\ \hline
(2 3),(1),(1) & (1),(1 3),(1) & (1),(1),(1 2) &
(2 3),(2 3),(2 3) & (1 3),(1 3),(1 3) & (1 2),(1 2),(1 2) \\ \hline
(2 3),(1 2),(1 3) & (1 2),(1 3),(2 3) & (1 3),(2 3),(1 2) &
(2 3),(1 2 3),(1 3 2) & (1 3 2),(1 3),(1 2 3) & (1 2 3),(1 3 2),(1 2) \\ \hline
(2 3),(1 3 2),(1 2 3) & (1 2 3),(1 3),(1 3 2) & (1 3 2),(1 2 3),(1 2) &
(2 3),(1 3),(1 2) & (2 3),(1 3),(1 2) & (2 3),(1 3),(1 2) \\ \hline

\hline
\end{tabular}
\medskip
\caption{Isomorphism classes of ternary quandles of order 3 }\label{terntable3}
\end{table}
\end{center}

\begin{center}
\begin{table}
\tiny
\begin{tabular}{|c|c|c||c|c|c|}

\hline
 z=1 & z=2 & z=3 &  z=1 & z=2 & z=3\\ \hline
(1),(1),(1) & (1),(2 3),(1) & (1),(1),(2 3) &
(1),(1),(1) & (1),(2 3),(2 3) & (1),(2 3),(2 3) \\ \hline
(1),(1),(1) & (2 3),(2 3),(1) & (2 3),(1),(2 3) &
(1),(1),(1) & (2 3),(2 3),(2 3) & (2 3),(2 3),(2 3) \\ \hline
(1),(1),(1) & (1 2 3),(1 2 3),(1 2 3) & (1 3 2),(1 3 2),(1 3 2) &
(1),(1),(1) & (1 3 2),(1 3 2),(1 3 2) & (1 2 3),(1 2 3),(1 2 3) \\ \hline
(1),(1),(1 2 3) & (1),(1 3 2),(1 3 2) & (1 2 3),(1 3 2),(1 2 3) &
(1),(1),(1 3 2) & (1 2 3),(1 3 2),(1 3 2) & (1 2 3),(1),(1 2 3) \\ \hline
(1),(2 3),(2 3) & (1),(2 3),(1) & (1),(1),(2 3) &
(1),(2 3),(2 3) & (1),(2 3),(2 3) & (1),(2 3),(2 3) \\ \hline
(1),(2 3),(2 3) & (2 3),(2 3),(1) & (2 3),(1),(2 3) &
(1),(2 3),(2 3) & (2 3),(2 3),(2 3) & (2 3),(2 3),(2 3) \\ \hline
(1),(2 3),(2 3) & (2 3),(1 2 3),(2 3) & (2 3),(2 3),(1 3 2) &
(1),(1 2),(1 3) & (1 3),(1 2 3),(1 2) & (1 2),(1 3),(1 3 2) \\ \hline
(1),(1 2 3),(1 2 3) & (1),(1 3 2),(1) & (1 3 2),(1 3 2),(1 2 3) &
(1),(1 2 3),(1 3 2) & (1),(1 2 3),(1 3 2) & (1),(1 2 3),(1 3 2) \\ \hline
(1),(1 2 3),(1 3 2) & (1 2 3),(1 3 2),(1) & (1 3 2),(1),(1 2 3) &
(1),(1 3 2),(1 2 3) & (1),(1 3 2),(1 2 3) & (1),(1 3 2),(1 2 3) \\ \hline
(1),(1 3 2),(1 2 3) & (1 3 2),(1 2 3),(1) & (1 2 3),(1),(1 3 2) &
(1),(1 3),(1 2) & (1 2),(1 2 3),(1 3) & (1 3),(1 2),(1 3 2) \\ \hline
(2 3),(1),(1) & (1),(2 3),(1) & (1),(1),(2 3) &
(2 3),(1),(1) & (1),(2 3),(2 3) & (1),(2 3),(2 3) \\ \hline
(2 3),(1),(1) & (2 3),(2 3),(1) & (2 3),(1),(2 3) &
(2 3),(1),(1) & (2 3),(2 3),(2 3) & (2 3),(2 3),(2 3) \\ \hline
(2 3),(1),(1) & (1 2 3),(2 3),(1 2 3) & (1 3 2),(1 3 2),(2 3) &
(2 3),(2 3),(2 3) & (1),(2 3),(1) & (1),(1),(2 3) \\ \hline
(2 3),(2 3),(2 3) & (1),(2 3),(2 3) & (1),(2 3),(2 3) &
(2 3),(2 3),(2 3) & (2 3),(2 3),(1) & (2 3),(1),(2 3) \\ \hline
(2 3),(2 3),(2 3) & (2 3),(2 3),(2 3) & (2 3),(2 3),(2 3) &
(2 3),(2 3),(2 3) & (1 2),(1 2),(1 2) & (1 3),(1 3),(1 3) \\ \hline
(2 3),(1 2),(1 3) & (2 3),(1 2),(1 3) & (2 3),(1 2),(1 3) &
(2 3),(1 2),(1 3) & (1 3),(2 3),(1 2) & (1 2),(1 3),(2 3) \\ \hline
(2 3),(1 2 3),(1 3 2) & (1),(2 3),(1 3 2) & (1),(1 2 3),(2 3) &
(2 3),(1 3 2),(1 2 3) & (1 3 2),(2 3),(1) & (1 2 3),(1),(2 3) \\ \hline
(2 3),(1 3),(1 2) & (1 2),(2 3),(1 3) & (1 3),(1 2),(2 3) &
(2 3),(1 3),(1 2) & (1 3),(1 2),(2 3) & (1 2),(2 3),(1 3) \\ \hline
(1 2),(1),(1 3 2) & (1 3 2),(2 3),(1) & (1),(1 3 2),(1 3) &
(1 2),(2 3),(1 3) & (1 2),(2 3),(1 3) & (1 2),(2 3),(1 3) \\ \hline
(1 2),(1 2),(1 2) & (2 3),(2 3),(2 3) & (1 3),(1 3),(1 3) &
(1 2),(1 2 3),(1 2 3) & (1 2 3),(2 3),(1 2 3) & (1 2 3),(1 2 3),(1 3) \\ \hline
(1 2),(1 3 2),(1) & (1),(2 3),(1 3 2) & (1 3 2),(1),(1 3) &
(1 2),(1 3),(2 3) & (1 3),(2 3),(1 2) & (2 3),(1 2),(1 3) \\ \hline
(1 2 3),(1),(1) & (1),(1 2 3),(1) & (1),(1),(1 2 3) &
(1 2 3),(1),(1 2 3) & (1 2 3),(1 2 3),(1) & (1),(1 2 3),(1 2 3) \\ \hline
(1 2 3),(1),(1 3 2) & (1 3 2),(1 2 3),(1) & (1),(1 3 2),(1 2 3) &
(1 2 3),(2 3),(1 3) & (1 2),(1 2 3),(1 3) & (1 2),(2 3),(1 2 3) \\ \hline
(1 2 3),(1 2),(1 2) & (2 3),(1 2 3),(2 3) & (1 3),(1 3),(1 2 3) &
(1 2 3),(1 2 3),(1) & (1),(1 2 3),(1 2 3) & (1 2 3),(1),(1 2 3) \\ \hline
(1 2 3),(1 2 3),(1 2 3) & (1 2 3),(1 2 3),(1 2 3) & (1 2 3),(1 2 3),(1 2 3) &
(1 2 3),(1 2 3),(1 3 2) & (1 3 2),(1 2 3),(1 2 3) & (1 2 3),(1 3 2),(1 2 3) \\ \hline
(1 2 3),(1 3 2),(1) & (1),(1 2 3),(1 3 2) & (1 3 2),(1),(1 2 3) &
(1 2 3),(1 3 2),(1 2 3) & (1 2 3),(1 2 3),(1 3 2) & (1 3 2),(1 2 3),(1 2 3) \\ \hline
(1 2 3),(1 3 2),(1 3 2) & (1 3 2),(1 2 3),(1 3 2) & (1 3 2),(1 3 2),(1 2 3) &
(1 2 3),(1 3),(2 3) & (1 3),(1 2 3),(1 2) & (2 3),(1 2),(1 2 3) \\
\hline
\end{tabular}
\medskip
\caption{Isomorphism classes of ternary $f$-quandles of order 3 }\label{terntable4}
\end{table}
\end{center}
\newpage
\normalsize
\section{Extensions of $f$-Quandles and Modules}

\noindent
In this section we investigate extensions of ternary $f$-quandles.  We define generalized ternary $f$-quandle $2$-cocycles and give examples.  We give an explicit formula relating group $2$-cocycles to ternary $f$-quandle $2$-cocycles, when the ternary $f$-quandle is constructed from a group.
\subsection{Extensions with dynamical cocycles and Extensions with constant cocycles}
\noindent \begin{proposition}Let $( X, T, F) $ be a ternary $f$-quandle and $A$ be a non-empty set. Let $\alpha: X \times X \times X \to \text{Fun}(A \times A \times A, A)$ be a function and $f, g: A \to A$ are maps.
Then,
$X \times  A$ is a ternary $f$-quandle by the operation $T((x,a), (y,b), (z,c))  = ( T(x,y,z), \alpha_{x,y,z}(a, b,c))$, where $T(x, y, z)$ denotes the ternary $f$-quandle product in $X$, if and only if $\alpha$ satisfies the following conditions:
\begin{enumerate}
\item $\alpha_{x,x,x}(a, a,a) = g(a)$ for all $x \in X$ and $ a \in A$;
\item $\alpha_{x, y, z}(-, b, c) : A \to A$ is a bijection for all $ x, y, z \in X$ and for all $ b, c  \in A$;
\item $\alpha_{T(x, y, z), f(u), f(v)}(\alpha_{x, y, z}(a, b, c), g(d), g(e)) = \\ \alpha_{T(x, u, v), T(y, u, v), T(z, u, v)}(\alpha_{x, u, v}(a, d, e), \alpha_{y, u, v}(b, d, e), \alpha_{z, u, v}(c, d, e))$ for all $ x, y, z, u, v \in X$ and $ a, b, c, d, e \in A$. \label{Tcocycle-3}
\end{enumerate}

\noindent Such function $\alpha$ is called a \textit{dynamical ternary $f$-quandle cocycle} or \textit{dynamical ternary $f$-rack cocycle} (when it satisfies above conditions).

\end{proposition}

\noindent The ternary $f$-quandle constructed above is denoted by $X \times_{\alpha} A$, and it  is called \textit{extension} of $X$ by a dynamical cocycle $\alpha$.  The construction is general, as Andruskiewitch and Gra$\tilde{n}$a showed in \cite{And-Grana}.\\

\noindent Assume $(X, T, F)$ is a ternary $f$-quandle and $\alpha$  be a dynamical $f$-cocycle. For $x \in X$, define $T_x(a, b, c) := \alpha_{x,x,x}(a, b, c)$. Then it is easy to see that $(A, T_x, F)$ is a ternary $f$-quandle for all $x \in X$.

\begin{remark}
When $x = y = z$ in condition (\ref{Tcocycle-3})  above, we get 
\begin{align*}
& \alpha_{f(x), f(u), f(v)}(\alpha_{x, x, x}(a, b, c), g(d), g(e)) = \\ & \hspace{3cm} \alpha_{T(x, u, v), T(x, u, v), T(x, u, v)}(\alpha_{x, u, v}(a, d, e), \alpha_{x, u, v}(b, d, e), \alpha_{x, u, v}(c, d, e))\end{align*} for all  $ a, b, c, d, e \in A$. 
\end{remark}

\noindent Now, we discuss Extensions with constant cocycles.
Let $(X, T, F)$ be a ternary $f$-rack and $\lambda: X \times X \times X \to S_A$ where $S_A$ is group of permutations of $X$.\\
If $\lambda_{T(x, y, z), F(u), F(v)} \lambda_{x, y, z} = \lambda_{T(x, u, v), T(y, u, v), T(z, u, v)}\lambda_{x, u, v}$ we say $\lambda$ is a \textit{constant ternary $f$-rack cocycle}.\\
If $(X, T, F)$ is a ternary $f$-quandle and further satisfies $\lambda_{x,x,x} = id$ for all $x \in X$ , then we say $\lambda$ is a \textit{constant ternary $f$-quandle cocycle}.

\subsection{Modules over Ternary $f$-rack}
\begin{definition}
Let $(X, T, F)$ be a ternary $f$-rack, $A$ be an abelian group and $f, g : X \to X$ be  homomorphisms. A structure of $X$-module on $A$ consists of a family of automorphisms  $(\eta_{ijk})_{i, j, k \in X}$ and a  family of  endomorphisms $(\tau_{ijk})_{i, j, k \in X}$ of  $A$ satisfying the following conditions:
\begin{eqnarray}
 &&\eta_{T(x, y, z), f(u), f(v)} \eta_{x, y, z} = \eta_{T(x, u, v), T(y, u, v), T(z, u, v)} \eta_{x, u, v} \label{Tmod-con-I} \label{E1}\\
 &&\eta_{T(x, y, z), f(u), f(v)} \tau_{x, y, z} = \tau_{T(x, u, v), T(y, u, v), T(z, u, v) }\eta_{y, u, v} \label{Tmod-con-II}\label{E2}\\
 &&\eta_{T(x, y, z), f(u), f(v)} \mu_{x, y, z} = \mu_{T(x, u, v), T(y, u, v), T(z, u, v)} \eta_{z, u, v} \label{Tmod-con-III} \label{E3}\\
 \label{Tmod-con-IV} &&\tau_{T(x, y, z), f(u), f(v)}g   = \eta_{T(x, u, v), T(y, u, v), T(z, u, v)} \tau_{x, u, v} + \tau_{T(x, u, v), T(y, u, v), T(z, u, v)}\tau_{y, u, v} \nonumber \\  && + \mu_{T(x, u, v), T(y, u, v), T(z, u, v)}\tau_{z, u, v}  \\
\label{Tmod-con-V} &&\mu_{T(x, y, z), f(u), f(v)}g = \eta_{T(x, u, v), T(y, u, v), T(z, u, v)}\mu_{x, u, v} +  \tau_{T(x, u, v), T(y, u, v), T(z, u, v)} \mu_{y, u, v} \nonumber \\ && +  \mu_{T(x, u, v), T(y, u, v), T(z, u, v)} \mu_{z, u, v} 
\end{eqnarray}
\end{definition}

\noindent In the $n$-ary case, we generalized the above definition as follows.\\
\begin{definition}
Let $(X, T, F)$ be an $n$-ary $f$-rack, $A$ be an abelian group and $f, g : X \to X$ be  homomorphisms. A structure of $X$-module on $A$ consists of a family of automorphisms  $(\eta_{ijk})_{i, j, k \in X}$ and a  family of  endomorphisms $(\tau^i_{ijk})_{i, j, k \in X}$ of  $A$ satisfying the following conditions:
\begin{align}
&\eta_{T(x_1,x_2,...,x_n),f(y_2),f(y_3),...,f(y_n)}\eta_{x_1,x_2,...,x_n} = 
\eta_{T(x_1,y2,...,y_n),T(x_2,y_2,...,y_n),...,T(x_n,y_2,...,y_n)}\eta_{x_1,y_2,...,y_n}\\
&\eta_{T(x_1,x_2,...,x_n),f(y_2),f(y_3),...,f(y_n)}\tau^i_{x_1,x_2,...,x_n} = 
\tau^i_{T(x_1,y2,...,y_n),T(x_2,y_2,...,y_n),...,T(x_n,y_2,...,y_n)}\eta_{x_i,y_2,...,y_n}\\
&\tau^i_{T(x_1,x_2,...,x_n),f(y_2),f(y_3),...,f(y_n)} g = \eta_{T(x_1,y2,...,y_n),T(x_2,y_2,...,y_n),...,T(x_n,y_2,...,y_n)}\tau^i_{x_1,y_2,...,y_n} \nonumber \\
&\hspace{4cm}+\sum_{j=1}^{n-1} \tau^j_{T(x_1,y2,...,y_n),T(x_2,y_2,...,y_n),...,T(x_n,y_2,...,y_n)}\tau^i_{x_j,y_2,...,y_n}.
\end{align}
\end{definition}

\begin{remark}\label{E5}
If $X$ is a ternary $f$-quandle, a ternary $f$-quandle structure of $X$-module on $A$ is a structure of an $X$-module further satisfies
\begin{eqnarray} \tau_{f(x), f(u), f(v)}g   &= (\eta_{T(x, u, v), T(x, u, v), T(x, u, v)} + \tau_{T(x, u, v), T(x, u, v), T(x, u, v)} + \nonumber \\ & \quad  \mu_{T(x, u, v), T(x, u, v), T(x, u, v)})\tau_{x, u, v} \end{eqnarray} and  \begin{eqnarray}  \mu_{f(x), f(u), f(v)}g & = (\eta_{T(x, u, v), T(x, u, v), T(x, u, v)} +  \tau_{T(x, u, v), T(x, u, v), T(x, u, v)}  \nonumber \\ & \quad +  \mu_{T(x, u, v), T(x, u, v), T(x, u, v)}) \mu_{x, u, v}. \end{eqnarray}

\noindent Furthermore, if $f, g =id$ maps, then it satisfies 
$$\eta_{T(x, u, v), T(x, u, v), T(x, u, v)} +  \tau_{T(x, u, v), T(x, u, v), T(x, u, v)}  +  \mu_{T(x, u, v), T(x, u, v), T(x, u, v)} = id.$$
\end{remark}

\begin{remark}
When $ x=y=z$ in \eqref{Tmod-con-I}, we get \begin{eqnarray*} \eta_{f(x), f(u), f(v)} \eta_{x, x, x}  = \eta_{T(x, u, v), T(x, u, v), T(x, u, v)} \eta_{x, u, v} .\end{eqnarray*}
\end{remark}


\begin{example} Let $A$ be a non-empty set and
$(X, T, F)$ be a ternary $f$-quandle, and $\kappa$ be a generalized $2$-cocycle. For $a, b, c \in A$, let\\
\begin{center}
$ \alpha_{x, y, z}(a, b, c) = \eta_{x, y,z}(a) + \tau_{x, y,z}(b) + \mu_{x, y, z}(c)+ \kappa_{x, y, z}$.
\end{center}
Then, it can be verified directly that $\alpha$ is a dynamical cocycle and the following relations hold:
\begin{eqnarray}
 &&\eta_{T(x, y, z), f(u), f(v)} \eta_{x, y, z} = \eta_{T(x, u, v), T(y, u, v), T(z, u, v)} \eta_{x, u, v} \\
 &&\eta_{T(x, y, z), f(u), f(v)} \tau_{x, y, z} = \tau_{T(x, u, v), T(y, u, v), T(z, u, v) }\eta_{y, u, v} \\
 &&\eta_{T(x, y, z), f(u), f(v)} \mu_{x, y, z} = \mu_{T(x, u, v), T(y, u, v), T(z, u, v)} \eta_{z, u, v} \\
&&\tau_{T(x, y, z), f(u), f(v)}g   = \eta_{T(x, u, v), T(y, u, v), T(z, u, v)} \tau_{x, u, v} + \tau_{T(x, u, v), T(y, u, v), T(z, u, v)}\tau_{y, u, v} \nonumber \\  && \hspace{3cm}+ \mu_{T(x, u, v), T(y, u, v), T(z, u, v)}\tau_{z, u, v}\\
&&\mu_{T(x, y, z), f(u), f(v)}g = \eta_{T(x, u, v), T(y, u, v), T(z, u, v)}\mu_{x, u, v} +  \tau_{T(x, u, v), T(y, u, v), T(z, u, v)} \mu_{y, u, v} \nonumber \\ && \hspace{3cm}+  \mu_{T(x, u, v), T(y, u, v), T(z, u, v)} \mu_{z, u, v} \\
&& \eta_{T(x, y, z), f(u), f(v)} \kappa_{x, y, z} + \kappa_{T(x, y, z), f(u), f(v)} = \eta_{T(x, u, v), T(y, u, v), T(z, u, v)} \kappa_{x, u, v}  \nonumber \\ && \hspace{3cm}+ \tau_{T(x, u, v), T(y, u, v), T(z, u, v)} \kappa_{y, u, v} + \mu_{T(x, u, v), T(y, u, v), T(z, u, v)} \kappa_{z, u, v}  \nonumber \\&& 
\hspace{3cm} +\kappa_{T(x, u, v), T(y, u, v), T(z, u, v)}\label{2-Tcocycle}.
 \end{eqnarray}
\end{example}
\quad 
\begin{definition}When $\kappa$ further satisfies $\kappa_{x,x,x} =0 $ in (\ref{2-Tcocycle}) for any $x \in X$, we call it a \textit{generalized ternary $f$-quandle $2$-cocycle}.\end{definition}
\noindent Recall that the \textit{quandle algebra} of an $f$-quandle $(X, \rt, f)$ is a $\mathbb{Z}$-algebra $\mathbb{Z}(X)$ presented by generators as in \cite{And-Grana} with relations  (\ref{E1}),(\ref{E2}), (\ref{E3}),( \ref{Tmod-con-IV}),(\ref{Tmod-con-V}),(13),(14).

\begin{example} \label{Tetaidtauzero}
Let $(X,T, F)$ be a ternary $f$-quandle and $A$ be an abelian group. \\ Set $\eta_{x, y, z} = \tau_{x, y, z} , \mu_{x, y, z} =0, \kappa_{x, y, z} = \phi(x, y, z) $.\\
Then $\phi$ is a $2$-cocycle. That is, \begin{eqnarray*}  \phi(x, y, z) + \phi(T(x, y, z), f(u), f(v))  =  \phi(x, u, v) + \phi(y, u, v)\\ \hspace{3cm} +  \phi(T(x, u, v), T(y, u, v), T(z, u, v)).\end{eqnarray*}
\end{example}

\begin{example} Let $\Gamma = \mathbb{Z}[P,Q,R]$ denote the ring of Laurent polynomials. Then any $\Gamma$-module $M$ is a $\mathbb{Z}(X)$-module for any ternary $f$-quandle $(X, T, F)$ by $\eta_{x, y, z}(a) = Pa$ ,$\tau_{x, y,z}(b) = Qb$  and $ \mu_{x, y, z}(c) = Rc$ for any $x, y, z \in X$.
\end{example}

\begin{definition}
A set $G$ equipped with a ternary operator $T : G \times G \times G \to G$ is said to be a ternary group $(G, T)$ if it satisfies the following condition:\\
\begin{enumerate}[(i)]
\item $T(T(x, y, z), u, v) = T(x, T(y, z, u), v) = T(x, y, T(z, u, v))$ (associativity),
\item $T(e, x, e) = T(x, e, e) = T(e, e, x) = x $ (existence of identity element),
\item $ T(x, y, y) = T(y, x, y) = T(y, y, x) = e $ ( existence of inverse element).
\end{enumerate}
\end{definition}

\begin{example}
Here we provide an example of a ternary $f$-quandle module and explicit formula of the ternary $f$-quandle $2$-cocycle obtained from a group $2$-cocycle.
Let $G$ be a group and let $ 0 \to A \to E \to G \to 1$ be a short exact sequence of groups where $E = A \rtimes_{\theta}   G$ by a group $2$-cocycle $\theta$ and $A$ is an Abelian group.\\
The multiplication rule in $E$ given by $(a, x) \cdot (b, y) \cdot (c, z)= ( a + x \cdot b + y \cdot c + \theta(x, y, z), T(x, y, z))$, where $x \cdot b$ means the action of $A$ on $G$. Recall that the group 3-cocycle condition is $$\theta(x, y, z) + x \theta(y, z, u) + \theta(x, yz, u) = \theta(xy, z, u) + \theta(x, y, zu).$$  Now, let $X=G$ be a ternary $f$-quandle with the operation $ T(x, y, z) = f(xy^{-1}z) $ and let $g: A \rightarrow A$ be a map on $A$ so that we have a map $F:E \rightarrow E$  given by $F(a,x)=(g(a), f(x))$.  Therefore the group $E$ becomes a ternary $f$-quandle with the operation $$ T((a,x),(b,y),(c, z)) = F( (a,x) \cdot (b, y)^{-1} \cdot (c,z)).$$  Explicit computations give that
 $\eta_{x, y, z}(a) = g(a)$, $\tau_{x, y, z}(b) = -2xy^{-1}g(b) $, $ \mu_{x, y, z} = y^{-1}g(c)$ and $\kappa_{x, y, z} = g[\theta(xy, y, y^{-1}) - \theta(xy^{-1}, y, y) + \theta(xy, y^{-1}, y) + 2\theta(x, y, e) - \theta(x, y^2, y^{-1}) - \theta(x, y^{-1}, y^2) + \theta(x, y^{-1}, y) - \theta(x, y,y) - \theta( x, y, y^{-1}) + \theta(x, y^{-1}, z)]$.

\end{example}

\section{Cohomology Theory of $n$-ary $f$-quandles}

In this section we present a general cohomology for $n$-ary $f$-quandles, and include specific examples, including the generalized ternary
case, and specific examples in both the ternary and binary case.\\

\noindent Let $(X, T, f)$ be a ternary $f$-rack where $f: X \to X$ is a ternary $f$-rack morphism. We will define the  generalized cohomology theory of $f$-racks as follows:\\
For a sequence of elements $(x_1, x_2, x_3, x_4,\dots, x_{2p+1}) \in X^{2p+1}$  define\\

\small
\begin{eqnarray*}
	&&[x_1, x_2, x_3, x_4,\dots, x_{2p+1}] =\\
	&&T(\dots T ( T ( T (x_1, x_2, x_3), f(x_4), f(x_5)), f^2(x_6),  f^2(x_7))) \dots ) f^{p-1}(x_{2p}) , f^{p-1}(x_{2p+1})).
\end{eqnarray*}
\normalsize

More generally, if we are considering an  $n$-ary $f$-rack $(X,T,f)$, using the same notation $T$ for the $n$-ary operation, we define the bracket as follows:

\small
\begin{eqnarray*}
	&&[x_1, x_2, x_3, x_4,\dots, x_{(n-1)p+1}] = T(\dots T (T(x_1, \dots, x_n), f(x_{n+1}), \dots, f(x_{2n-1})) \dots\\
	&&\dots, f^{p-1}(x_{p(n-2)+1}),\dots,f^{p-1}(x_{(n-1)p+1}))
\end{eqnarray*}
\normalsize

\noindent Notice that for $ i=(p-1)j+1 <n$, we have 

\begin{eqnarray*}
	&&[x_1, x_2, x_3, x_4,\dots, x_n] =\\
	&&T([x_1,\dots, x_{i-1},x_{i+p},\dots, x_n], f^{i-2}[x_i,x_{i+p},\dots, x_n],f^{i-2}[x_{i+1},x_{i+p},\dots, x_n],\dots \\
	&&\dots,f^{i-2}[x_{i+p-1},x_{i+p},\dots, x_n] \\
\end{eqnarray*}
\normalsize

\noindent This relation is obtained by applying the first axiom of $f$-quandles $p-i$ times, first grouping the first $i-1$ terms together, then iterating this process, again grouping and iterating each.

\noindent We provide cohomology theory for the $f$-rack

\begin{theorem} \label{Cohomology}
	Consider the free left $\mathbb Z(X)$-module $C_p(X) = \mathbb Z(X)X^p $ with basis $X^p$.  For an abelian group $A$, denote $C^p(X,A):= Hom_{\mathbb Z(X)}(C_p(X) ,A)$, then the coboundary operators are defined as  $\partial = \partial^p : C^{p+1}(X) \to C^p(X)$ such that  	
	\begin{eqnarray*}
		\lefteqn{
			\partial^p \phi (x_1, \dots, x_{(n-1)p+1}) } \nonumber \\ && =
		(-1)^{p+1}\sum_{i=2}^{p+1} (-1)^{i} \{ \eta_{[A(i)],F^{i-2}([B_2(i)]),F^{i-2}([B_3(i)]),...,F^{i-2}([B_{n}(i)])} \phi (A(i))
		\nonumber \\
		&&
		- \phi  (T(C_1(i)),T(C_2(i)),...,T(C_{(n-1)i-1}(i)), F(x_{(n-1)i}),F(x_{(n-1)i+1}),...,F(x_{(n-1)p+1}))
		\nonumber\\
		&&
		+ (-1)^{p+1} \sum_{j=1}^{n-1} \tau^{i}_{[B_1(0)],[B_2(0)],...,[B_{n-1}(0)]} \phi (B_j(0)),
	\end{eqnarray*}
	
	where
	$A(i)=x_1,x_2,...,\hat{x}_{(n-1)i},\hat{x}_{(n-1)i+1},...,\hat{x}_{(n-1)(i+1)},x_{n(i+1)-i},...,x_{(n-1)p+1},$\\
	$B_k(i)=x_{(n-1)i+k},x_{(n-1)i+n+1},x_{(n-1)i+2},...,x_{(n-1)p+1},$\\
	$C_k(i)=x_k,x_{(n-1)i},x_{(n-1)i+1},...,x_{(n-1)i+n-2}.$
	
	Then the pair $(C(X), \partial)$ defines a cohomology complex.
\end{theorem}


\begin{proof}
	
	\noindent To prove that $\partial^{p+1}\partial^{p}=0$, and thus $\partial$ is a coboundary map we will break the composition into pieces, using the linearity of $\eta$ and $\tau^i$.
	
	\noindent First we will show that the composition of the $i^{th}$ term of the first summand of $\partial^p$ with the $j^{th}$ term of the first summand of $\partial^{p+1}$ cancels with the $(j+1)^{th}$ term of the first summand of $\partial^p$ with the $i^{th}$ term of the first summand of $\partial^{p+1}$ for $i\leq j$. As the sign of these terms are opposite, we need only show that the compositions are equal up to their sign. For the sake of readability we will introduce the following, based on $A$ and $B$ above:
	\begin{eqnarray*}A(i,j) && =x_1,...,\hat{x}_{(n-1)i},\hat{x}_{(n-1)i+1},...,\hat{x}_{ni},x_{ni+1},...,x_{(n-1)j-1}, \\&& \quad \hat{x}_{(n-1)j}, \hat{x}_{(n-1)j+1},...,\hat{x}_{nj},...,x_{(n-1)p+1},\end{eqnarray*}
	\begin{eqnarray*} B(i,j)=x_{(n-1)i+k},x_{ni+1},x_{ni+2},...,\hat{x}_{(n-1)j},\hat{x}_{(n-1)j+1},...,\hat{x}_{nj},...,x_{(n-1)p+1}.\end{eqnarray*}

	Now,  we can see that the composition of the $i^{th}$ term of the first summand of $\partial^p$ with the $j^{th}$ term of the first summand of $\partial^{p+1}$ can be rewritten as follows:
	
	\begin{multline*}
	\eta_{[A(i)],F^{i-2}[B_0(i)],F^{i-2}[B_1(i)],...,F^{i-2}[B_{n-2}(i)]}
	\eta_{[A(i,j+1)],F^{j-1}[B_0(j+1)],F^{j-1}[B_1(j+1)],...,F^{j-1}[B_{n-2}(j+1)]}
	\\
	=\eta_{T([A(i,j)],F^{j-1}[B_0(j+1)],...,F^{j-1}[B_{n-2}(j+1)]),T(F^{i-2}[B_0(i,j)],F^{j-1}[B_0(j+1)],...,F^{j-1}[B_{n-2}(j+1)])}
	\\
	_{T(F^{i-2}[B_1(i,j)],F^{j-1}[B_0(j+1)],...,F^{j-1}[B_{n-2}(j+1)]),...,T(F^{i-2}[B_{n-2}(i,j)],F^{j-1}[B_0(j+1)],...,F^{j-1}[B_{n-2}(j+1)])}
	\\
	\eta_{[A(i,j+1)],F^{j-1}[B_0(j+1)],F^{j-1}[B_1(j+1)],...,F^{j-1}[B_{n-2}(j+1)]}
	\\
	=\eta_{T([A(i,j)],F^{i-2}[B_0(i,j)],...,F^{i-2}[B_{n-2}(i,j)]),F^{j-1}[B_0(j+1)],...,F^{j-1}[B_{n-2}(j+1)]}
	\eta_{[A(i,j)],F^{i-2}[B_0(i,j)],...,F^{i-2}[B_{n-2}(i,j)]}
	\\
	=\eta_{[A(j+1)],F^{j-1}[B_0(j+1)],F^{j-1}[B_1(j+1)],...,F^{j-1}[B_{n-2}(j+1)]}\eta_{[A(i,j)],F^{i-2}[B_0(i,j)],...,F^{i-2}[B_{n-2}(i,j)]}.
	\end{multline*}

	This is precisely the $(j+1)^{th}$ term of the first summand of $\partial^p$ with the $i^{th}$ term of the first summand of $\partial^{p+1}$.
	
	Similar manipulations show that the composition of $\tau^i$ from $\partial^p$ with the $i^{th}$ term of the first sum of $\partial^{p+1}$ cancels with the composition of the $(i+1)^{th}$ term of the first sum of $\partial^p$ with $\tau^i$ from $\partial^{p+1}$. For the sake of brevity we will omit showing these manipulations, but the table below presents all relations which are canceled by similar manipulations.\\
	
	In the table, $\eta_i$ represents the $i^{th}$ summand of the first sum, $\circ_i$ represents the $i^{th}$ summand of the second sum, with order of composition determining its origin in $\delta_p$ or $\delta_{p+1}$.
	\begin{center}
		\begin{tabular}{ c c c }
			$\eta_i \eta_j$ & $=$ & $\eta_{j+1}\eta_i$ \\
			$\eta_i \circ_j$ & $=$ & $\circ_{j+1} \eta_i$ \\
			$\eta_i \tau^i$ & $=$ & $\tau^i \eta_{i+1}$ \\
			$\tau^i \circ_i $ & $=$ &  $ \circ_{i+1} \tau^i $\\
			$\circ_i \circ_j$ & $=$ & $\circ_{j+1} \circ_i$ .\\
		\end{tabular}
	\end{center}
	
	\noindent All these relations leave $n+1$ remaining terms, which cancel via the third axiom from the Definition.
\end{proof}

\noindent We present the ternary case below, using the convention from the previous section, so $\tau$ and $\mu$ representing $\tau^1$ and $\tau^2$ respectively.

\begin{example}\label{n=2}
By specializing $n=2$ in  Theorem \ref{Cohomology}, the coboundary operator simplifies to:
%
	\begin{eqnarray*}
		\lefteqn{
			\partial \phi (x_1, \dots, x_{2p+1}) } \nonumber \\ && =
		\sum_{i=1}^{p} (-1)^{i}    \eta_{\{A, B, C\}} \phi (x_1, \dots, \hat{x}_{2i}, \hat{x}_{2i+1}, \dots, x_{2p+1})
		\nonumber \\
		&&
		- \sum_{i=1}^{p} (-1)^{i}   \phi  (T(x_1,x_{2i}, x_{2i+1}), \dots, T(x_{2i-1},x_{2i},x_{2i+1}),F(x_{2i+2}), \dots, F(x_{2p+1}))
		\nonumber\\
		&&
		+ (-1)^{2p+1}   \tau_{[x_{1}, x_{4}, \dots, x_{2p+1}],[x_{2}, x_4, \dots, x_{2p+1}], [x_{3}, \dots, x_{2p+1}]} \phi (x_{2}, x_{4},\dots, x_{2p+1})\nonumber\\
		&&
		+ (-1)^{2p+1}   \mu_{[x_{1}, x_{4}, \dots, x_{2p+1}],[x_{2}, x_4, \dots, x_{2p+1}], [x_{3}, \dots, x_{2p+1}]} \phi (x_{3}, x_{4},\dots, x_{2p+1}),
	\end{eqnarray*}
	\noindent where $A = [x_1, \dots, \hat{x}_{2i},\hat{x}_{2i+1}, \dots, x_{2p+1}]$,
	$B= F^{\{i-1\}}[x_{2i},x_{2i+2}, x_{2i+3}, \dots, x_{2p+1}] $,\\
	$C=  F^{\{i-1\}}[x_{2i+1},x_{2i+2}, \dots, x_{2p+1}]$.\
\end{example}
Specializing further in  Example \ref{n=2}, we obtain the following result.

\begin{example} 
	In this example, we compute the first and second cohomology groups of the ternary Alexander $f$-quandle $X=\mathbb{Z}_3$ with coefficients in the abelian group $\mathbb{Z}_3$. For the ternary $f$-quandle under consideration we have $P=2$, $Q=R=1$, that is $T(x_1,x_2,x_3)=P x_1 + Q x_2 + R x_3$ and $f(x)=(P+Q+R)x$ as in Example \ref{TAlex}.  Now, setting $\eta$ to be the multiplication by $P$, $\tau$ to be the multiplication by $Q$, and $\mu$ to be the multiplication by $R$,  we have the $1$-cocycle condition for $\phi : X \to A$ given by 
	
	$$P\phi(x)+Q\phi(y)+R\phi(z)-\phi(T(x,y,z))=0$$
	
	and the 2-cocycle condition as
	
	\begin{eqnarray*}
		\lefteqn{ P\psi(x_1,x_2,x_3)+\psi(T(x_1,x_2,x_3),f(x_4),f(x_5))}\\
		& = & P\psi(x_1,x_4,x_5)+Q\psi(x_2,x_4,x_5)+R\psi(x_3,x_4,x_5)\\
		& + &\psi(T(x_1,x_4,x_5),T(x_2,x_4,x_5),T(x_3,x_4,x_5)).
	\end{eqnarray*}
	
	A direct computation gives $H^1(X=\mathbb{Z}_3,A=\mathbb{Z}_3)$ is $2$-dimensional with basis $\{2 \chi_0+\chi_1,2\chi_0+\chi_2\}$. As such the $dim(Im(\delta_1))=1$, and additional calculation gives $dim(ker(\delta_2))=3$, thus $H^2$ is also $2$-dimensional.
\end{example}

Lastly we consider a binary case, obtaining, as expected, a familiar result.

\begin{example} Let $\eta$ be the multiplication by $T$ and $\tau$ be the multiplication by $S$  in  Example \ref{TAlex}.	The $1$-cocycle condition is written for a function $\phi: X \rightarrow A$
	as
	$$ T \phi(x) + S \phi(y) - \phi(x* y)=0.$$
	Note that this means that $\phi: X \rightarrow A$ is a quandle homomorphism.

	For $\psi: X \times X \rightarrow A$, 	the $2$-cocycle condition can be written     as
	
	\begin{eqnarray*}
		\lefteqn{ T  \psi (x_1, x_2) +   \psi (x_1 * x_2, f(x_3)) } \\
		& = & T  \psi (x_1, x_3) + S \psi (x_2, x_3)
		+  \psi (x_1 * x_3, x_2 * x_3).
	\end{eqnarray*}
	
In \cite{CEGM}, the groups $H^1$ and $H^2$ with coefficients in the abelian group $\mathbb{Z}_3$ of the $f$-quandle  $X =\mathbb{Z}_3$,  $T=1,  S = 2$ and  $f(x) = 0$ were computed.   More precisely, $H^1(\mathbb{Z}_3,\mathbb{Z}_3) $ is $1$-dimensional with  a basis  $ \chi_1+2 \chi_2$ and $H^2$ is $1$-dimension with a basis  $\{\chi_{(1,2)} - \chi_{(2,1)}\}.$
	
\end{example}

\end{document}